\documentclass{article}
\usepackage{amsmath, amsfonts, amssymb}

\newtheorem{theorem}{Theorem}[section]
\newtheorem{lemma}[theorem]{Lemma}
\newtheorem{proposition}[theorem]{Proposition}

\newtheorem{remark}[theorem]{Remark}

\begin{document}
\title{Multiplicity of solutions for gradient systems under strong resonance at the first eigenvalue}
\author{
Edcarlos D. da Silva \\ IME-UFG,  , Goi\^ania, Brazil \\ \textsf{edcarlos@mat.ufg.br} \\
 }
\date{}
\maketitle
$\mathbf{Abstract:}$ In this paper we establish existence and multiplicity of solutions
for an elliptic system which has strong resonance at first eigenvalue. To describe the resonance, we use an eigenvalue
problem with indefinite weight. In all results we use Variational Methods.

$\mathbf{Keywords:}$ Strong Resonance, Variational Methods, Indefinite Weights.

\section{Introduction}

In the present paper we discuss results on the existence and multiplicity of solutions for the system
\begin{equation}\label{p}
 \text{ }
 \left\{ \begin{array}[c]{cc}
     - \triangle u  = a(x)u + b(x)v - f(x, u , v) \, \, \mbox{in} \, \, \Omega & \\
     - \triangle v  = b(x)u + d(x)v - g(x, u , v)  \, \, \mbox{in} \, \, \Omega  & \\
     u = v = 0 \, \, \mbox{on}\, \, \partial \Omega, &\\
  \end{array}
\right.
\end{equation}
where $\Omega \subseteq \mathbb{R}^{N}$ is bounded smooth domain in $\mathbb{R}^{N}$, $N \geq 3$ with
$a,b,d \in C^{0}(\overline{\Omega}, \mathbb{R})$ and $f, g \in C^{1}(\overline{\Omega} \times \mathbb{R}^{2}, \mathbb{R})$.
Moreover, we assume that there is some function $F \in C^{2}(\overline{\Omega} \times \mathbb{R}^{2}, \mathbb{R})$ such that $\nabla F = (f, g)$.
 Here and throughout this paper, $\nabla F$ denotes the gradient in the variables $u$ and $v$.  Under this hypotheses, the problem
 \eqref{p} is clearly variational of the gradient type. Indeed, it is a system which has been studied by many authors, see
 \cite{BC,costa3, Odair} and references therein.

On the other hand, the system \eqref{p} represents a steady state case of reaction-diffusion systems of interest
in biology, chemistry, physics and ecology. Mathematically, reaction-diffusion systems take the form of nonlinear
 parabolic partial differential equations which
have been intensively studied during recent years, see \cite{Smoller, Pao} where many references can be found.

From a variational stand point, finding weak solutions of \eqref{p} in $H = H^{1}_{0}(\Omega) \times H^{1}_{0}(\Omega)$
is equivalent to finding critical points of the $C^{2}$ functional given by
\begin{equation}\label{J}
J(z) = \frac{1}{2} \|z\|^{2} - \dfrac{1}{2}\int_{\Omega} \langle A(x)(u,v), (u,v)\rangle dx + \int_{\Omega}F(x,u,v)dx,
\end{equation}
with  $z = (u, v) \in H$.

We work with system \eqref{p} where occurs strong resonance at infinity. More specifically, we assume
strong resonance conditions using an eigenvalue problem with weights given by the linear part of the system \eqref{p}.
Moreover, we consider the resonance at first eigenvalue.

Now, we introduce our eigenvalue problem with weights. Let us denote by $\mathcal{M}_{2}(\Omega)$ the set of all
continuous, cooperative and symmetric matrices $A$ of order 2, given by

$$A(x)= \left(\begin{array}{cc}
  a(x) & b(x) \\
  b(x) & d(x) \\
\end{array} \right),$$
where the functions $ a,b,d \in C(\overline{\Omega}, \mathbb{R})$
satisfy the following hypotheses:
\begin{flushleft}
$\left(M_{1}\right)$ A is cooperative, that is, $b(x)\geq 0$.
 \end{flushleft}
\begin{flushleft}
$\left(M_{2}\right)$ There is $x_{0} \in \Omega$ such that $a(x_{0}) > 0$ or $d(x_{0}) > 0.$
\end{flushleft}

Given $A \in \mathcal{M}_{2}(\Omega)$, consider the weighted
eigenvalue problem

\begin{equation}\label{LPi}
\left\{\begin{array}{c}
  - \triangle \left(\begin{array}{c}
       u \\
       v \\
      \end{array}\right)  = \lambda A(x)\left(\begin{array}{c}
       u \\
       v \\
      \end{array}\right)  \, \, \mbox{in} \, \,
      \Omega \\
  u = v = 0 \, \, \mbox{on}\, \, \partial \Omega.\\
\end{array} \right.
\end{equation}

Now, using the conditions $(M_{1})$ and $(M_{2})$ above, and applying the spectral theory for compact operators,
 we get a sequence of eigenvalues

$$  0 < \lambda_{1}(A) <  \lambda_{2}(A) \leq \lambda_{3}(A) \leq \ldots $$
such that $\lambda_{k}(A)\rightarrow + \infty$ as $ k \rightarrow
\infty$ see \cite{Ch,DG,Odair}. Here, each eigenvalue $\lambda_{k}(A), k \geq 1$ has finite multiplicity.

Next, we state the assumptions and the main results in this paper. First, we make the following hypothesi:
\begin{flushleft}
$(SR)$ There exist $h \in L^{1}(\Omega)$ such that
\end{flushleft}
\begin{equation}
\lim_{|z| \rightarrow \infty} \nabla F(x,z) = 0 \, \mbox{and} \,\, |F(x,z)|
\leq h(x), \, \mbox{a.e.} \, x \in \Omega,\, \forall \, z \in \mathbb{R}^{2}.
\end{equation}
In this way, we define the following functions:
\begin{eqnarray}\label{au}
  T^{+}(x) = \liminf_{u \rightarrow \infty \atop v \rightarrow \infty } F(x, u,v),& & S^{+}(x) =
  \limsup_{u \rightarrow \infty \atop v \rightarrow \infty } F(x, u,v), \nonumber \\
  T^{-}(x) = \liminf_{u \rightarrow - \infty \atop v \rightarrow - \infty } F(x, u,v),& & S^{-}(x) =
  \limsup_{u \rightarrow -\infty \atop v \rightarrow - \infty } F(x, u,v). \nonumber \\
\end{eqnarray}
Here, the functions above define functions in $L^{1}(\Omega)$ and the limits are take a. e. and uniformly in $x \in \Omega$.
Clearly, we have $T^{-}(x) \leq T^{+}(x)$ and $S^{-}(x) \leq S^{+}(x), \, \mbox{a.e} \, \, x \, \in \Omega.$

Now, we consider the additional hypotheses:

\begin{flushleft}
$(H1)F(x, z) \geq \dfrac{1}{2}(1 - \lambda_{2})\langle A(x)z, z \rangle + b_{1}|\Omega|^{-1}, \forall \, (x, z) \in \Omega \times \mathbb{R}^{2}$.
\end{flushleft}

\begin{flushleft}
$(H2) \langle A(x)z, z \rangle \geq 0, \forall \, (x, z) \in \Omega \times \mathbb{R}^{2}.$
\end{flushleft}
Thus, we can prove that the associated functional $J$ has the saddle geometry. Hence, we have the following result:
\begin{theorem}\label{t1}
Suppose $(SR),(H1), (H2)$. Then problem (\ref{p}) has at least one solution $z_{1} \in H$.
\end{theorem}

Now, we take $\nabla F(x,0,0) \equiv 0,F(x,0, 0) \equiv 0$ and $h_{1}=h_{2} \equiv 0$.
Then the problem \eqref{p} admits the trivial solution $(u,v) \equiv  0$. In this case, the main
 point is to assure the existence of nontrivial solutions. The existence of these solutions depends
 mainly on the behaviors of $F$ near the origin and near infinity. Thus, we consider the following additional hypotheses:

\begin{flushleft}
$(H3)$ There exist $\alpha \in (0 , \lambda_{1})$ and $\delta > 0$
such that  $$F(x,z) \geq  \dfrac{1 - \alpha}{2} \langle A(x)z , z
\rangle, \forall \, x \in \Omega \, \,\mbox{and} \,\, |z| < \delta.$$
\end{flushleft}

\begin{flushleft}
$(H4) \int_{\Omega} S^{+}(x)dx \leq 0 \,\, \mbox{and} \, \,
\int_{\Omega} S^{-}(x)dx \leq 0. $
\end{flushleft}

\begin{flushleft}
$(H5)$ There exist $t \in \mathbb{R}^{*}$ such that
$$\int_{\Omega} F(x, t \Phi_{1}) dx < \min \left\{\int_{\Omega} T^{-}(x)dx, \int_{\Omega} T^{+}(x)dx \right\}.$$
\end{flushleft}
In this way, applying the Ekeland's Variational Principle and the Mountain Pass Theorem, we can prove the following multiplicity results:
\begin{theorem}\label{t2}
Suppose $(SR), (H2), (H3), (H4)$ and $(H5)$. Then problem (\ref{p}) has at least two nontrivial solutions.
\end{theorem}

For the last result, we minimize the functional under some subsets on $H$. In this case, we consider the following additional hypothesi:

\begin{flushleft}
$(H6)$ There are $t^{-} < 0 < t^{+} $ such that
$$\int_{\Omega} F(x, t^{\pm}\Phi_{1}) dx < \min \left\{\int_{\Omega} T^{-}(x)dx, \int_{\Omega} T^{+}(x)dx \right\}$$
\end{flushleft}

Hence, combining the ideas developed in Theorem \eqref{t2} we have the following multiplicity result:

\begin{theorem}\label{t3}
Suppose $(SR),(H1),(H2), (H3), (H4)$ and $(H6)$. Then problem (\ref{p}) has at least
three nontrivial solutions.
\end{theorem}

In our main theorems we consider the case when the functions defined by  \eqref{au} are nonpositive. Indeed, we have some interesting
geometries produced some multiplicity results.  For the case where the functions in \eqref{au} are positive are treat by similar arguments.
Thus, with further hypotheses, we have at least three nontrivial solutions for problem \eqref{p}. More specifically, we have two solutions
with negative energy and one solution produced by Theorem \ref{t1}. We leave the details for the reader.

 Now, we compare our results with the previous results in the literature. Most of previous results treated problem \eqref{p} using variational
 methods, sub-super solutions method or degree theory, see \cite{BC,Ch, Ch3, Odair} and references therein. In these works, the authors proved
 several results on existence and multiplicity for problem \eqref{p}. In paper \cite{Ch},
  K.C. Chang  consider the problem \eqref{p} with nonresonance conditions using variational methods and Morse theory. In paper \cite{BC},
  T. Bartsch, K.C. Chang and Z. Q. Wang  obtained sign-solutions under resonant conditions. They used the conditions of the Ahmad, Lazer and
  Paul type introduced in \cite{A}.  In paper \cite{Ch3}, K.C. Chang consider the problem \eqref{p} using sub-super solution method and degree
   theory. In paper \cite{Odair}, Furtado and de Paiva used the non-quadraticity condition at infinity and Morse theory.
    However, little has been done to for the resonant case. For example, the strong resonance case do not considered.

In this article, we explore the strong resonance case at the first eigenvalue. For this case, we prove the functional $J$ has an interesting
geometry under some hypotheses on $F$. Thus, we obtain different results on existence and multiplicity of solutions for problem \eqref{p} combining some min-max theorems which complement previous results in the literature.

The paper is organized as follows: in Section 2, we recall the abstract framework of problem (\ref{p}) and highlight the properties for the
eigenvalue problem (\ref{LPi}). In section 3 we prove some auxiliary results involving the Palais-Smale condition and some
properties on the geometry for the functional $J$. In Section 4 is devoted to the proofs our main Theorems.

\section{Abstract Framework and Eigenvalue Problem  for the System (\ref{p})}\label{eigenvalue}
Initially, we write $H = H^{1}_{0}(\Omega) \times H^{1}_{0}(\Omega)$ to denote the Hilbert space with the norm
$$\|z\|^{2} = \int_{\Omega} |\nabla u|^{2} + |\nabla v|^{2} dx, z = (u , v) \in H.$$
Moreover, we denote by $ \langle , \rangle $ the scalar product in $H$.

Again, we remember the properties of the eigenvalue problem \\
\begin{equation}\label{LP}
\left\{\begin{array}{c}
  - \triangle \left(\begin{array}{c}
       u \\
       v \\
      \end{array}\right)  = \lambda A(x)\left(\begin{array}{c}
       u \\
       v \\
      \end{array}\right)  \, \, \mbox{in} \, \,
      \Omega \\
  u = v = 0 \, \, \mbox{on}\, \, \partial \Omega.\\
\end{array} \right.
\end{equation}

Let $A \in \mathcal{M}_{2}(\Omega)$, then there is only a compact
self-adjoint linear operator $T_{A} : H \rightarrow H$ such that: $\langle T_{A} z, w \rangle = \int_{\Omega} < A(x)z ,w > dx, \forall z,
w \in H.$ This operator has the following propriety: $\lambda$ is eigenvalue of (\ref{LP}) if, and only if, $T_{A} z = \frac{1}{\lambda}
z,$ for some $z \in H$. Thus, for each matrix $A \in \mathcal{M}_{2}(\Omega)$ there exist a sequence of eigenvalues for system (\ref{LP}) and
a Hilbertian basis for $H$ formed by eigenfunctions of (\ref{LP}).
Moreover, denoting by $\lambda_{k}(A)$ the eigenvalues of problem (\ref{LP}) and $\Phi_{k}(A)$ the associated eigenfunctions, then $0
 < \lambda_{1}(A) < \lambda_{2}(A) \leq  \ldots  \leq \lambda_{k}(A) \rightarrow \infty$ if $k \rightarrow \infty$, and we have
$$ \dfrac{1}{\lambda_{k}(A)} = \sup \{ \langle T_{A} z, z\rangle, \|z\| = 1, z \in V_{k-1}^{\perp}\},$$
where $ V_{k - 1}^{\perp}= span\{\Phi_{1}(A), \ldots,\Phi_{k}(A)\}$ with $k > 1$. Thus, we get $H = V_{k}\bigoplus V_{k}^{\perp}$
for $k \geq 1$, and the following variational inequalities holds:

\begin{equation}\label{v0}
\|z\|^{2} \geq \lambda_{1}(A) \langle T_{A}z,z \rangle, \forall \, z \in H,
\end{equation}

\begin{equation}\label{v1}
\|z\|^{2} \leq \lambda_{k}(A) \langle T_{A}z,z \rangle, \forall \, z \in V_{k}, k \geq 1,
\end{equation}

\begin{equation}\label{v2}
\|z\|^{2} \geq \lambda_{k + 1}(A) \langle T_{A}z,z \rangle, \forall \, z \in V_{k}^{\perp}, k \geq 1.
\end{equation}
The variational inequalities will be used in the next section. Now, we would like to mention that the eigenvalue $\lambda_{1}(A)$
is positive, simples and isolated. Moreover, we have that the associated eigenfunction $\Phi_{1}(A)$ is positive in $\Omega$.
In other words, we have a Hess-Kato theorem for eigenvalue problem \eqref{LP} proved by Chang, see \cite{Ch}.
For more properties to the eigenvalue problem (\ref{LP}) see \cite{Ch3, DG,Odair} and references therein.

\section{Preliminary Results}

In this section we prove some results needed in the proof of our main theorems.
First, we prove the Palais-Smale condition at some levels for the functional $J$. Then we describe some results under the geometry for $J$.

First, we recall that $J: H \rightarrow \mathbb{R}$ is
said to satisfy Palais-Smale condition at the level $c \in \mathbb{R}$
((PS)$_{c}$ in short), if any sequence $ (z_{n})_{n \in \mathbb{N}} \subseteq H$ such that
$$J(z_{n}) \rightarrow c \, \, \mbox{and} \,\, J^{'}(z_{n}) \rightarrow 0 $$
as $n \rightarrow \infty $, possesses a convergent subsequence in $H.$ Moreover, we say that $J$ satisfies $(PS)$ condition when we have
$(PS)_{c}$ for all $c \in \mathbb{R}$.

\begin{lemma}\label{l1}
Suppose $(SR)$. Then the functional $J$ has the $(PS)_{c}$ condition whenever $c < \min \left\{\int_{\Omega} T^{-}(x)dx, \int_{\Omega}
T^{+}(x)dx \right\}$ or $c > \max \left\{\int_{\Omega} S^{-}(x)dx, \int_{\Omega} S^{+}(x)dx \right\}$.
\end{lemma}

\begin{proof}
Initially, we divide the proof this lemma in two parts. In first part, we prove this result with $c > \max \left\{\int_{\Omega} S^{-}(x)
dx, \int_{\Omega} S^{+}(x)dx \right\}$. Obviously, the second part treat the case where  $c < \min \left\{\int_{\Omega} T^{-}(x)dx,
\int_{\Omega} T^{+}(x)dx \right\}$.

Now, we prove the first part. The proof is by contradiction. In this case, we suppose that there exist a $(PS)_{c}$ unbounded sequence
 $(z_{n})_{n \in \mathbb{N}} \in H$ such that $c > \max \left\{ \int_{\Omega} S^{+}(x)dx, \int_{\Omega} S^{-}(x)dx \right\}$. Thus,
 we obtain the following informations:
\begin{itemize}
    \item $J(z_{n}) \rightarrow c,$
    \item $\|z_{n}\| \rightarrow \infty,$
    \item $\|J^{'}(z_{n})\| \rightarrow 0,  \,\,\mbox{as} \, \,  n \rightarrow \infty$.
\end{itemize}
Thus, we define $\overline{z_{n}} = \dfrac{z_{n}}{\|z_{n}\|}$. So, there is $\overline{z} \in H$ with the following properties:
\begin{itemize}
    \item $\overline{z_{n}} \rightharpoonup \overline{z} \,\, \mbox{em}\,H,$
    \item $\overline{z_{n}} \rightarrow \overline{z}\,\, \mbox{em}\, L^{p}(\Omega)^{2},$
    \item $\overline{z_{n}}(x) \rightarrow \overline{z}(x) \, \mbox{a. e. in} \, \Omega.$
\end{itemize}

On the other hand, we easily see that $\overline{z} = \pm \Phi_{1}$. So, we suppose initially that
$\overline{z} = \Phi_{1}$. Then, we have $\overline{u_{n}}(x)
\rightarrow \infty$ and $\overline{v_{n}}(x) \rightarrow \infty,
\, \, \forall \, x \in \Omega $ as $n \rightarrow \infty$. Here, we use that
$\Phi_{1} > 0$ in $\Omega$.

Hence, we write $z_{n} =
t_{n}\Phi_{1} + w_{n}$, where $(t_{n})_{n \in \mathbb{N}} \in \mathbb{R}$ and $(w_{n})_{n \in \mathbb{N}} \in
V_{1}^{\perp}$. Thus, we obtain the following inequality:
\begin{equation}\label{109r}
J(z_{n}) \geq \dfrac{1}{2} \left(1 - \dfrac{1}{\lambda_{2}(A)}\right)\|w_{n}\|^{2} + \int_{\Omega} F(x, z_{n})dx.
\end{equation}
But, the limitation on $F$ and the inequality \eqref{109r} imply that
$|t_{n}| \rightarrow \infty$ as $n\rightarrow \infty$. Moreover, we have that $\|w_{_{n}}\| \leq C, \forall \, n \in \mathbb{N}$.
For see this, we suppose that $(w_{n})$ is unbounded. Thus, using the inequality \eqref{109r} we obtain $J(z_{n}) \rightarrow \infty$
 as $n \rightarrow \infty$.Thereofore, we have  a contraction. Consequently, $(w_{n})_{n \in \mathbb{N}}$ is a sequence bounded in $H$
 and the sequence $(t_{n})_{n \in \mathbb{N}} \in \mathbb{R}$ is unbounded.

Now, using H$\ddot{o}$lder's inequality and Sobolev's embedding and $(SR)$ we have the following estimates:
\begin{equation}\label{107r}
\left | \int_{\Omega} \nabla F(x,z_{n}) w_{n} dx \right| \leq C\left(\int_{\Omega} |\nabla F(x,z_{n})|^{2})dx \right)^{\frac{1}{2}}
 \|w_{n}\| \leq C \left(\int_{\Omega} |\nabla F(x,z_{n})|^{2} dx \right)^{\frac{1}{2}}.
\end{equation}
Thus, applying the o Dominated Convergence Theorem we conclude the following identity:
\begin{equation}\label{108r}
 \lim_{n \rightarrow \infty } \int_{\Omega} \nabla F(x,z_{n}) w_{n} dx = 0.
\end{equation}

Now, using \eqref{107r} and \eqref{108r}, we have
\begin{eqnarray}\label{110r}
  \left(1 - \dfrac{1}{\lambda_{2}}\right) \|w_{n}\|^{2} &\leq& \left|  \|w_{n}\|^{2} - \langle T_{A} w_{n}, w_{n} \rangle - \int_{\Omega}
  \nabla F(x, z_{n})w_{n}dx \right| + \left|\int_{\Omega} \nabla F(x, z_{n})w_{n} dx\right| \nonumber \\
   &\leq& \dfrac{1}{n}\|w_{n}\| + \left|\int_{\Omega} \nabla F(x, z_{n})w_{n} dx\right| \leq \dfrac{1}{n}\|w_{n}\| + \dfrac{1}{n}, \forall
    \, n \in \mathbb{N}. \nonumber \\
\end{eqnarray}
Therefore, by \eqref{110r}, we conclude that $\|w_{n}\| \rightarrow 0$, as $n \rightarrow \infty$. So, using Sobolev's embedding we obtain that
\begin{equation}\label{111r}
    \|w_{n}\|^{2} - \langle T_{A} w_{n}, w_{n}\rangle \rightarrow 0, \, \mbox{se} \,\, n \rightarrow \infty.
\end{equation}

On the other hand, we have the following identity
$$ J(z_{n}) = \dfrac{1}{2}\|z_{n}\|^{2} - \dfrac{1}{2}\int_{\Omega} \langle A(x)z_{n} , z_{n} \rangle dx + \int_{\Omega} F(x,z_{n})dx.$$
Consequently, we have that
\begin{eqnarray}\label{10}
  c &=& \lim_{n \rightarrow \infty}J(z_{n}) = \limsup_{n \rightarrow \infty} \left\{ \dfrac{1}{2}\|z_{n}\|^{2} - \dfrac{1}{2}\int_{\Omega}
  \langle A(x)z_{n} , z_{n} \rangle dx + \int_{\Omega} F(x,z_{n})dx \right\} \nonumber \\
&=& \limsup_{n \rightarrow \infty} \left\{ \dfrac{1}{2}\|w_{n}\|^{2} - \dfrac{1}{2}\int_{\Omega}
  \langle A(x)w_{n} , w_{n} \rangle dx + \int_{\Omega} F(x,z_{n})dx \right\} \nonumber \\
 &=& \limsup_{n \rightarrow \infty} \int_{\Omega} F(x,z_{n}) dx  = \limsup_{n \rightarrow \infty} \int_{\Omega} F(x,t_{n}\Phi_{1} + w_{n}) dx
  = \limsup_{n \rightarrow \infty} \int_{\Omega} F(x,t_{n}\Phi_{1}) dx \nonumber \\
  &\leq&  \int_{\Omega} \limsup_{n \rightarrow \infty}F(x,t_{n}\Phi_{1})dx \nonumber =  \int_{\Omega} S^{+}(x)dx,  \nonumber\\
\end{eqnarray}
where we used \eqref{111r}, Fatou's Lemma and $(SR)$. Finally, we have a contradiction because we choose initially $c \in \mathbb{R}$ such
that $c > \max \left\{\int_{\Omega} S^{+}(x)dx, \int_{\Omega} S^{-}(x)dx \right \}$. The case where $\overline{z} = - \Phi_{1}$ is treated
by similar arguments. We leave the details for the reader. Therefore, the functional $J$ satisfy the $(PS)_{c}$ condition for all $c >
\max \left\{\int_{\Omega} S^{+}(x)dx, \int_{\Omega}S^{-}(x)dx \right \}.$

Now, we consider the second part in the proof of this theorem. In this case, we can prove that the functional $J$ satisfy the $(PS)_{c}$
condition whenever $c < \max \left\{\int_{\Omega} T^{+}(x)dx, \int_{\Omega}
T^{-}(x)dx \right \}$. Again, we consider the same ideas developed in first part. So, we omit the details in this case.
\end{proof}

For the next result we prove that functional $J$ has the saddle geometry given by Theorem 1.11 in \cite{S}. Thus, we prove the following result:

\begin{proposition}\label{p1}
Suppose $(SR)$ and $(H1)$. Then the functional $J$ has the following saddle geometry:
\begin{description}
    \item[a)] $I(z) \rightarrow \infty$ if $\|z\|\rightarrow \infty$ with $z \in V_{1}^{\perp}.$
    \item[b)] There is $\alpha \in \mathbb{R}$ such that $I(z) \leq \alpha, \, \, \forall \,  \, z \in V_{1}$.
    \item[c)] $I(z) \geq b_{1}, \, \forall \, z \in V_{1}^{\perp}.$
\end{description}
\end{proposition}
\begin{proof}
Initially, the proof of the cases $a)$ and $b)$ are standard. In these cases, we use variational inequality \eqref{v2} and the limitation on $F$.
So, we leave the proof of the cases $a)$ and $b)$ for the reader.

Now, we prove the case $c)$. For this case, using $(H1)$ and the variational inequality \eqref{v2}, we have the following estimates:
\begin{eqnarray}
  J(z) &=& \dfrac{1}{2}\|z\|^{2} - \dfrac{1}{2}\langle T_{A} z, z \rangle + \int_{\Omega} F(x,z) dx \geq \dfrac{1}{2}\|z\|^{2} -
  \dfrac{\lambda_{2}}{2}\langle T_{A} z, z \rangle +  b_{1} \nonumber \\
   &\geq&  \dfrac{1}{2}\left(1 - \dfrac{\lambda_{2}}{\lambda_{2}}\right)\|z\|^{2} + b_{1} = b_{1}, \forall \, z \in V_{1}^{\perp}. \nonumber \\
\end{eqnarray}
Therefore, we obtain the inequality given in $c)$. So, we finish the proof of this theorem.
\end{proof}

Now, with hypotheses describe in this note, we prove that functional $J$ has the Mountain Pass Theorem. The arguments
used in this results are standard.

\begin{proposition}\label{p2}
Suppose $(SR)$ and $(H3)$. Then the origin is a local minimum for the functional $J$.
\end{proposition}

\begin{proof}
           First, using $(H3)$, we can choose $p \in (2, 2^{*})$ and a constant $C_{\epsilon}
           > 0$ such that
           $$F(x,z) \geq \dfrac{1 - \alpha}{2} \langle A(x)z, z \rangle - C_{\epsilon} |z|^{p}, \forall \, (x,z) \in \Omega \times
           \mathbb{R}^{2}.$$
           Consequently, we have the following estimates
           \begin{eqnarray}
             J(z) &=& \dfrac{1}{2}\|z\|^{2} - \dfrac{1}{2}\langle A(x)z, z \rangle + \int_{\Omega} F(x,z)dx \nonumber
             \geq \dfrac{1}{2}(1 - \dfrac{\alpha}{\lambda_{1}})\|z\|^{2} - C_{\epsilon}\int_{\Omega}|z|^{p}dx  \nonumber\\
              &\geq& \dfrac{1}{2}\left(1 - \dfrac{\alpha}{\lambda_{1}}\right)\|z\|^{2} -
              C_{\epsilon}\|z\|^{p} \geq \dfrac{1}{4}\left(1 - \dfrac{\alpha}{\lambda_{1}}\right)\|z\|^{2} >
             0, \nonumber \\
           \end{eqnarray}
           where $z \in B_{\rho}(0)\backslash \{ 0 \}$ and $0 < \rho \leq \rho_{0}$ with $\rho_{0}$ small enough. Here, $B_{\rho}(0)$
           denote the open ball with center in the origin and radius $\rho$ in $H$. Therefore,
           the proof of this propositions it follows.
\end{proof}

For complete the Mountain Pass geometry, we prove the following result:

\begin{proposition}\label{p3}
Suppose $(SR), (H2)$ and $(H5)$. Then there exist $z \in H$ such that
$I(z) < 0$ where $\|z\| > \rho_{0} > 0.$
\end{proposition}

\begin{proof}
Firstly, using $(H5)$ and $(H2)$, we take $z = t \Phi_{1}$ where $t \in \mathbb{R}^{*}$ is provided by $(H5)$ . Thus, we obtain the
following estimates:
\begin{eqnarray}
  J(t \Phi_{1}) &=& \dfrac{1}{2}\|t \Phi_{1}\|^{2} - \dfrac{1}{2}\int_{\Omega}\langle A(x)t \Phi_{1}, t \Phi_{1}\rangle dx +
   \int_{\Omega} F(x, t \Phi_{1})dx \nonumber   \\
   &=& \int_{\Omega} F(x, t \Phi_{1})dx  < \min \left(\int_{\Omega} T^{-}x)dx, \int_{\Omega} T^{+}x)dx \right) \leq 0.  \nonumber  \\
\end{eqnarray}
Therefore, we have that  $J(t \Phi_{1}) < 0$ with $\|t_{0} \Phi_{1}\| = |t_{0}| >
\rho_{0}$. Here, we remember that the first eigenfunction satisfy $\|\Phi_{1}\| =1 $. So, the proof of this proposition it follows.
\end{proof}

Next, we prove that problem \eqref{p} has at least one solution using the Ekeland's Variational Principle. In this case, the
key point is assure that
the infimun of $J$ satisfy the Palais-Smale condition.

\begin{proposition}\label{p6}
Suppose $(SR),(H4)$ and $(H5)$.  Then problem \eqref{p} has at least one nontrivial solution $z_{0} \in H$. Moreover, the
solution $z_{0}$ has negative energy.
\end{proposition}

\begin{proof}
First, we remember that the function $F$ is bounded. Therefore, the functional $J$ is bounded bellow.
In this case, we would like to mention that $J$ has the $(PS)_{c}$ condition with $c = \inf \{J(z) : z \in H \}$. For see this,
we use the Lemma \ref{l1} and we take $t \in \mathbb{R}^{*}$ provided by $(H5)$. So, we obtain the following estimates:
\begin{eqnarray}
   c &\leq& J(t\Phi_{1}) = \int_{\Omega}F(x, t \Phi_{1})dx < \min \left\{\int_{\Omega} T^{+}(x)dx,\int_{\Omega}T^{-}(x)dx
    \right \} \leq 0. \nonumber  \\
\end{eqnarray}

Consequently, applying the Ekeland Variational Principle we have one critical point $z_{0} \in H$ such that
 $J(z_{0}) = \inf \{J(z) : z \in H \} \leq J(t_{0} \Phi_{1}) < 0$. Thus, $z_{0}$ satisfy $J(z_{0}) < 0$.
 Therefore, the problem (\ref{p}) has at least one nontrivial solution. This affirmation concludes the proof of this theorem.
\end{proof}

For the next results we find another solutions for the problem \eqref{p} by minimization on some
subsets of $H$. More specifically, we define following subsets:
$$A^{+} = \{ t \Phi_{1} + w, t \geq 0, \, w \in V_{1}^{\perp}\},$$
$$A^{-} = \{ t \Phi_{1} + w, t \leq 0, \, w \in V_{1}^{\perp}\}.$$
Thus, we have $ \partial A^{+} = \partial A^{-} = V_{1}^{\perp}$.
So, we minimizer the functional $J$ restrict to $A^{+}$ and $A^{-}$.

\begin{proposition}\label{p4}
Suppose $(SR), (H1) $ and $(H6)$. Then problem \eqref{p} has at least two nontrivial solutions with negative energy.
\end{proposition}

\begin{proof}
First, we consider the functionals $J^{\pm} =
J|_{A^{\pm}}$. Thus, we have that $J^{\pm}$
has $(PS)_{c}$ condition whenever $c < \min \left\{\int_{\Omega}
T^{+}(x)dx, \int_{\Omega}T^{-}(x)dx \right\}$, see Lemma \ref{l1}. Therefore, we obtain that $J^{\pm}$ satisfy the $(PS)_{c^{\pm}}$
condition with $c^{\pm} = \inf \left\{ J^{\pm}(z) : z \in H \right\}$.

In this way, applying the Ekeland Variational Principle for the functional $J^{\pm}$ we obtain two critical points which we denote by
 $z_{0}^{+}$ and $z_{0}^{-}$, respectively. Thus, we have the following informations:
$$ c^{+} = J^{+}(z_{0}^{+}) = \inf_{z \in A^{+}}J(z) \, \, \,\mbox{and} \, \,\, c^{-} = J^{-}(z_{0}^{-}) = \inf_{z \in A^{-}}J(z).$$

Moreover, we afirme that $z_{0}^{+}$ and $z_{0}^{-}$ are nonzero critical points. For see this, we use $(H4)$ and $(H6)$
obtaing the following estimates:
\begin{eqnarray}\label{12r}
& & J^{\pm}(z_{0}^{\pm}) \leq J(t^{\pm}\Phi_{1}) = \int_{\Omega}F(x,t^{\pm} \Phi_{1})dx < \min \left\{\int_{\Omega} T^{+}(x)dx,\int_{\Omega}
T^{-}(x)dx \right\} \leq 0. \nonumber  \\
\end{eqnarray}

On the other hand, using  $(H6)$, we obtain that $J$ restrict to $V_{1}^{\perp}$ is nonnegative. More specifically, give
$w \in V_{1}^{\perp}$ we have the following estimates:
\begin{eqnarray}\label{13r}
  J(w) &=& \dfrac{1}{2}\|w\|^{2} - \dfrac{1}{2}\int_{\Omega} \langle A(x)w , w \rangle dx +
    \int_{\Omega} F(x,w)dx \nonumber \\
  &\geq& \dfrac{1}{2}\|w\|^{2} - \dfrac{\lambda_{2}}{2} \int_{\Omega} \langle A(x)w , w \rangle dx \geq 0, \nonumber   \\
\end{eqnarray}
where we use the variational inequality \eqref{v2}.

Now, we prove that $z_{0}^{+}$ and $z_{0}^{-}$ are distinct. The proof of this affirmation is by contradiction. In this case, we suppose
that $z_{0}^{+}=z_{0}^{-} \in V_{1}^{\perp}$. Then, using the estimate \eqref{13r} we obtain that $J(z_{0}^{\pm}) < 0 \leq J(z_{0}^{\pm})$.
 Therefore,  we have a contradiction. Consequently, we obtain that $z_{0}^{+} \neq z_{0}^{-}$. Thus, $z_{0}^{\pm}$ are distinct critical
 points and the problem \eqref{p} has at least two nontrivial solutions. Moreover, these solutions has negative energy, see \eqref{12r}.
 This affirmation concludes the proof of this proposition.
\end{proof}

\section{Proof of the main Theorems}

$\mathbf{Proof  \; of \; the \; Theorem \, \, \ref{t1}}$
Initially, we have the $(PS)_{c}$ condition for some levels $c \in \mathbb{R}$ given by Lemma \ref{l1}.
 Thus, we take $H = V_{1} \bigoplus V_{1}^{\perp},$ where $V_{1} = span \{\Phi_{1}\}$. So, using Proposition
 \ref{p1} we conclude that the functional $J$ has saddle point geometry given by Theorem $1.11$ in \cite{S}.
 Therefore, we have one critical point $z_{1} \in H$ for $J$. This statement finish the proof of this theorem.

$\mathbf{Proof  \; of \; the \; Theorem \, \, \ref{t2}}$
First, using Propositions \ref{p2} and \ref{p3} we have the mountain pass geometry for the functional $J$.
 Moreover, the functional $J$ has the $(PS)_{c}$ condition for all $c \geq 0$, see Lemma \ref{l1}. Thus, we
  have a solution $z_{2} \in H$ given by the Mountain Pass Theorem. Obviously, the solution $z_{2}$ satisfies $J(z_{2}) > 0$.

On the other hand, using the Proposition \ref{p6} we obtain one solution $z_{0}$ such that $J(z_{0}) < 0$.
 Consequently, we have $z_{0} \neq z_{2}$ and the Problem \eqref{p} has at least two nontrivial solutions.
 This affirmation concludes the proof of this theorem.

$\mathbf{Proof  \; of \; the \; Theorem \, \, \ref{t3}}$
 In this case, we use the Propositions \ref{p2} and \ref{p3} getting a mountain pass point $z_{2}$ such that
  $J(z_{2}) > 0$. Moreover, using the Proposition \ref{p4} we obtain two critical points $z_{0}^{\pm}$ such
  that $J(z_{0}^{\pm})< 0$. Therefore, we obtain that $z_{2},z_{0}^{\pm}$ are distinct critical points. So,
  the problem \eqref{p} has at leas three nontrivial solutions. This statements finish the proof of this theorem.

 \begin{remark}
In our main theorems, we allow that the functions defined in \eqref{au} to be equal. In this case, we define a function $w \in L^{1}(\Omega)$
 such that $w(x) = \lim_{|z| \rightarrow \infty} F(x,z)$. In particulary, we prove that the functional $J$ satisfies the $(PS)_{c}$ condition
  for each $c \in \mathbb{R}\backslash \int_{\Omega} w(x)dx$. Moreover, the functional $J$ do not satisfies the $(PS)_{c}$ condition for
   $c = \int_{\Omega} w(x)dx$.
\end{remark}

$\mathbf{Acknowledgment}$ The author thanks Professor Djairo G. de Figueiredo for his encouragement, comments and helpful conversations.


\end{document}